\documentclass[11pt]{article}        

\usepackage{amsmath,amssymb,amsthm,hyperref,theoremref,lipsum}
\providecommand{\keywords}[1]{\textbf{\textbf{Keywords}:} #1}
\providecommand{\subjclass}[1]{\textbf{\textbf{Mathematics subject classification(2010)}:} #1}
\theoremstyle{plain}
\newtheorem{thm}{Theorem}[section]

\newtheorem{cor}{Corollary}[section]
\newtheorem{rmk}{Remark}[section]
\newtheorem{lem}{Lemma}[section]

\numberwithin{equation}{section}
\title{Angular changes of complex Fourier coefficients of cusp forms}

\author{Mohammed Amin Amri}
\newcommand{\Addresses}{{
  \bigskip
  \footnotesize

  Mohammed Amin.~Amri, \textsc{ACSA Laboratory, Department of Mathematics,Faculty of Sciences, Mohammed First University,Oujda, Morocco}\par\nopagebreak
  \textit{E-mail address}, Mohammed Amin~Amri: \texttt{amri.amine.mohammed@gmail.com}

}}

\def\re{{\Re e\,}}
\def \im{{\Im m\,}}

\newcommand{\C}{\mathbb C}
\newcommand{\Q}{\mathbb Q}

\newcommand{\N}{\mathbb N}
\newcommand{\R}{\mathbb R}
\newcommand{\W}{\mathcal W}
\begin{document}

\maketitle

\begin{abstract}
In this paper, we investigate the ``angular changes" behavior of some subfamilies of Fourier coefficients of both integral and half-integral weight holomorphic cusp forms, thus one gets information about signs of the real an imaginary parts of these subfamilies. These give an extension of some recent results of Kohnen and his collaborators.
\end{abstract}



\section{Introduction}
The sign changes problem of the Fourier coefficients of cusp forms has been the focus of much recent study, due to their various number theoretic applications. Coming back to the general scenario, in \cite{murty83} Ram Murty, proved that for an arbitrary cusp form belonging to any congruence subgroup, either the real or the imaginary parts of the subsequence of its Fourier coefficients at primes numbers changes sign infinitely often. After that, there has been more extensive study of the Fourier coefficients of other kinds of automorphic forms.

Among many other results, Knopp, Kohnen and Pribitkin in \cite{knopp03} show that the real and the imaginary parts of Fourier coefficients of cusp forms of positive real weight, with multiplier system, changes sign infinitely often.  Going further in this direction, in \cite{kohnen14} Kohnen and Martin proved that the subsequence of Fourier coefficients supported on prime power indices of an even integral weight normalised Hecke eigenform for the full modular group change sign infinitely often. 

The question about the sign changes of Fourier coefficients of half-integral weight modular forms had been asked by Bruiner and Kohnen \cite{bruin}, and there it was shown that the subsequence $\{a(tn^2)\}_{n\ge1}$ of Fourier coefficients of half-integral weight cusp forms has infinitely many  sign changes when a certain $L$--function has no zeros in the interval $(0,1)$, later, in \cite{kohnen10}, this hypothesis has been removed.   

In this paper firstly, we extend the result \cite[Theorem 2.1]{kohnen14} of Kohnen and Martin for an even integral weight normalised newform of arbitrary level $N$ with Dirichlet character $\chi\pmod N$ (see \thref{thm2} for a precise statement). 
We also, extend the result obtained by Kohnen in \cite[Theorem]{kohnen10} for a half-integral weight cusp forms on $\Gamma_0(4N)$, with not necessarily real Dirichlet character $\chi\pmod{4N}$ contained in the orthogonal complement of the subspace of $S_{k+1/2}(4N,\chi)$ generated by the unary theta functions (see \thref{thm3} for a precise statement). Finally, we generalise \cite[Theorem 2.2]{bruin} of Kohnen and Bruiner, for a half-integral weight Hecke eigenforms on $\Gamma_0(4N)$, with not necessarily real Dirichlet character (see \thref{thm4} for a precise statement).

The proofs of the theorems are following broadly the same lines as the proofs of the corresponding conditional results shown in \cite{kohnen14,kohnen10,bruin}.  The essential ingredients are a reformulation in terms of ``wedge'' of Fekete's extension of Landau's theorem \cite{Fekete}. Deligne's theorem \cite[Theorem 8.2]{Deligne73}, and analytic properties of Hecke $L$-functions attached to cusp forms.

\section{Statements of results}
To set up the notations, let $k,N\in\N$ be integers, we denote by $S_k(N,\chi)$ the space of holomorphic cusp forms of weight $k$ and level $N$, with Dirichlet character $\chi\pmod N$, we denote by $r_{\chi}$, the order of the Dirichlet character $\chi$. We call a ``wedge" the portion of the plane given by 
$$
\mathcal{W}(\theta_1,\theta_2):=\{re^{i\theta}\; : \; r\ge 0,\;\; \theta\in[\theta_1,\theta_2]\},
$$
with $0\le\theta_2-\theta_1<\pi$. Our definition is slightly different from that in \cite{huls16}, since if $r>0$, then escaping from the ``wedge" is as easy as having a zero value. Our first main result is the following.
\begin{thm}\thlabel{thm2}
Let $f\in S_k(N,\chi)$, be a normalized newform of even integral weight $k$ and level $N$, with Dirichlet character $\chi$, and assume that $r_{\chi}$ is odd. Let
$$
f(z)=\sum_{n\ge 1}a(n)e(nz),
$$
be the Fourier expansion of $f$ at $\infty$. Let $j\ge1$ be an integer not divisible by $2$. Then for almost all primes $p$ the sequence $\{a(p^{nj})\}_{n\in\N}$ escape infinitely often from the wedge $\mathcal{W}(\theta_1,\theta_2)$.
\end{thm}
As a consequence of \thref{thm2}, we have the following corollary.

\begin{cor}
For for almost all primes $p$ either $\{\re (a(p^{nj}))\}_{n\in\N}$ or $\{\im (a(p^{nj}))\}_{n\in\N}$ changes sign infinitely often.
\end{cor}

In our next results, we study the ``angular changes" behavior of Fourier coefficients of holomorphic cusp forms of half-integral weight. Before stating our results, we need introduce some notations. Let $N\ge4 $ be divisible by $4$, write $S_{k+1/2}(N,\chi)$ for the space of holomorphic cusp forms of half-integral weight $k+1/2$ and level $N$ with character $\chi\pmod N$. From the work of Shimura \cite{shimura73} we know that $S_{k+1/2}(N,\chi)$ can contain single-variable theta-series for $k=1$, let $U(N,\chi)$ be the subspace generated by unary theta functions. If $k\ge2$ then $U(N,\chi)=0$. But this is often not the case for $k=1$. We put $S^{*}_{k+1/2}(N,\chi):=U^{\perp}(N,\chi)$, the orthogonal complement of $U(N,\chi)$ with respect to the Petersson inner product. We shall prove the following. 

\begin{thm}\thlabel{thm3}
Let $f\in S^{*}_{k+1/2}(N,\chi)$ be a cusp form of half integral weight $k+1/2$, level $N$, with Dirichlet character $\chi$, let
$$
f(z)=\sum_{n\ge1}a(n)e(nz),
$$
be its Fourier expansion at $\infty$. Let $t$ be a square-free natural number, suppose  there is $n_0$  such that $a(tn_0^2)\neq0$. Then the sequence $\{a(tn^2)\}_{n\ge1}$  escape infinitely often from the wedge $\mathcal{W}(\theta_1,\theta_2)$.
\end{thm}
Under the hypotheses of \thref{thm3} we have

\begin{cor}
Either $\{\re (a(tn^2))\}_{n\ge1}$ or $\{\im (a(tn^2))\}_{n\ge1}$ changes sign infinitely often.
\end{cor}

\begin{thm}\thlabel{thm4}
Let $f\in S^{*}_{k+1/2}(N,\chi)$ be a Hecke eigenform, of half integral weight $k+1/2$, and level $N$, with Dirichlet character $\chi$. Assume that $r_{\chi^2}$ is odd and let
$$
f(z)=\sum_{n\ge1}a(n)e(nz),
$$
be the Fourier expansion of $f$ at $\infty$. Let $t$ be a square free natural number. Then, for almost all primes $p$, the sequence $\{a(tp^{2\nu})\}_{\nu\in\N}$  escape infinitely often from the wedge $\mathcal{W}(\theta_1,\theta_2)$.
\end{thm}
Under the hypotheses of \thref{thm4} we obtain the following result.

\begin{cor}
For almost all primes $p$ either $\{\re (a(tp^{2\nu}))\}_{\nu\in\N}$ or $\{\im (a(tp^{2\nu}))\}_{\nu\in\N}$  changes sign infinitely often.
\end{cor}

\section{Proofs}
We begin with the following crucial result of Deligne \cite[Theorem 8.2]{Deligne73}.
\begin{thm}[Deligne's Theorem]\thlabel{deligne}
Let $f(z)=\sum_{n\ge1} a(n)e(nz),$ $a(1) = 1,$ be a newform of integral-weight $k$ and level $N$ with Dirichlet character $\chi$. Then for each prime $p$ not dividing $N$, we have
\[
1-a(p)p^{-s}+\chi(p)p^{k-1-2s}=(1-\alpha_p p^{-s})(1-\beta_p p^{-s}),
\]
and $|\alpha_p| =|\beta_p| =p^{\frac{k-1}{2}}$. In particular
\[
a(n)\ll_\varepsilon n^{\frac{k-1}{2}+\varepsilon},
\]
for every $\varepsilon>0$.	
\end{thm}

At this point we state, the following theorem, which would play a crucial role in the rest of this paper.
\begin{thm}
 
\thlabel{thm1}
Let $\W(\theta_1,\theta_2)$ be an arbitrary wedge, and let 

\begin{equation}
L(s)=\sum_{n=1}^{\infty}\frac{a(n)}{n^s},\label{eq2}
\end{equation}
be a Dirichlet series whose coefficients lie inside the wedge $\mathcal{W}(\theta_1,\theta_2)$ for all but finitely many $n\ge1$,  assume further that its abscissa of convergence $\sigma_c$ is finite, then \eqref{eq2} has a singularity at $s=\sigma_c$ (\eqref{eq2} cannot be continued analytically beyond the line $\re(s)=\sigma_c$).
\end{thm}

\begin{rmk}\thlabel{rmk1}

\begin{itemize}
\item This theorem is a reformulation of \cite[Theorem 2]{maurizi11,Fekete} in terms of wedges, it provides a criterion to have infinitely many ``angular changes" (or escaping from a ``wedge") of the coefficients of the Dirichlet series. It can be derived from \cite[Theorem 2]{maurizi11} by making the following observation, since multiplication of \eqref{eq2} by $e^{i\phi}$ for any $\phi\in\mathbb{R}$, does not affect neither the hypothesis nor the conclusion, hence we may assume that $a(n)=|a(n)|e^{\phi_n}$ lies in $\mathcal{W}(\phi,-\phi)$, where $\phi=\frac{\theta_2-\theta_1}{2}\in[0,\frac{\pi}{2})$, therefore we have $\cos(\phi_n)\ge \gamma$, for all $n\in\mathbb{N}$,  where $\gamma=\cos(\phi)>0$.
\item Notice that \thref{thm1} implies that 
$$
\sigma_{\mathrm{c}}=\sigma_{\mathrm{ab}}=\sigma_{\mathrm{hol}},
$$
where $\sigma_{\mathrm{ab}}$, $\sigma_{\mathrm{c}}$, $\sigma_{\mathrm{hol}}$, denotes respectively the abscissa of absolute convergence, the abscissa of convergence, the abscissa of holomorphy of the Dirichlet series in question.

\end{itemize}
\end{rmk}

\subsection{Proof of \texorpdfstring{\thref{thm2}}{Theorem 2}}

In this subsection, we prove \thref{thm2}. In order to do this, we first define a family of operators on the space $S_k(N,\chi)$. Let $f\in S_k(N,\chi)$ be a cusp form, then it admits a Fourier expansion at $\infty$ of the form 
$$
f(z)=\sum_{n\ge1}a(n)e(nz).
$$
For each non-negative integer $j$ and a prime $p$, we define the action of the operator $T_j(p)$ on $f$ by 
\begin{equation}
T_j(p)f(z)= \sum_{n\ge1}\left(a(p^{j}n)+p^{j(k-1)}\chi^{j}(p)a\left(\frac{n}{p^{j}}\right)\right)e(nz),\label{eq,4}
\end{equation}
with the usual convention $a(n/p^{j})=0$ if $p^j$ does not divides $n$. We should note that $T_0(p)=2$ and $T_1(p)=T(p)$ where $T(p)$ is the $p$-th classical Hecke operator. We will need the following lemma.

\begin{lem}\thlabel{lem}
Let $p$ be a prime number and $j\ge1$ an integer. The following assertions hold.

\begin{enumerate}
\item $T_j(p)$ is a monic polynomial in $T(p)$ of degree $j$.
\item If $f\in S_k(N,\chi)$ is an eigenfunction of $T_j(p)$ with eigenvalue $\lambda_j(p)$, then

\begin{equation}
\sum_{n\ge0}a(p^{jn})X^n=\dfrac{a(1)}{1-\lambda_j(p)X+p^{j(k-1)}\chi^j(p)X^2}\cdot\label{eq.1}
\end{equation}
 where $a(n)$ denote the $n$-th Fourier coefficient of $f$.
\end{enumerate}
\end{lem}

\begin{proof}[Proof of \thref{lem}]
\begin{enumerate}
\item  We see easily from \eqref{eq,4} that for all $j\ge 1$ one has
$$
T_{j+1}(p)=T_j(p)T(p)-p^{k-1}\chi(p)T_{j-1}(p),
$$
hence the result follows by recurrence on $j$.
\item Let $n\in\N$. To prove \eqref{eq.1}, it suffices to show that
$$
a(p^{j(n+1)})=\lambda_j(p)a(p^{jn})-p^{j(k-1)}\chi^j(p)a(p^{j(n-1)}),
$$
for all $j\ge1$, which can be deduced from \eqref{eq,4}.
\end{enumerate}
\end{proof}
Now we are in position to prove \thref{thm2}, we shall follow closely the method of Kohnen and Martin in \cite[Proof of Theorem 2.1]{kohnen14}. 

Let $p$ be a prime, $ p\nmid N$,  for which $a(p^{jn})$ lies in the wedge $\mathcal{W}(\theta_1,\theta_2)$ for all but finitely many $n\ge 0$. Then the Dirichlet series 

\begin{equation}
\sum_{n\ge0}a(p^{jn})p^{-jns}\quad (\re(s)\gg 1),\label{eq:2}
\end{equation}
satisfies the hypothesis of \thref{thm1}. Hence two situations can occur, either (a) the series  has a pole on the real point of its line of convergence or (b) it converges for all $s\in\C$. We will disprove the assertion (a) for all but a finite number of primes $p$ and disprove (b) for all $p$. We start by considering the first case (a).

Let $\zeta:=e^{2\pi i/j}$ be a primitive $j$-th root of unity and let $\mu\in\mathbb{Z}$. A similar argument to that in \cite[Proof of Theorem 2.1]{kohnen14} yields 
\begin{equation}
\sum_{n\ge0}a(p^{jn})p^{-jns}=\frac{1}{j}\sum_{\mu=0}^{j-1}\dfrac{1}{(1-\zeta^{\mu}\alpha_p p^{-s})(1-\zeta^{\mu}\beta_p p^{-s})}\quad (\re(s)\gg1)\cdot\label{eq:5}
\end{equation}
By our hypothesis, one of the denominators on the right-hand side of \eqref{eq:5}
has a real zero. In this case necessarily at least one of the numbers $\alpha_p\zeta^{\mu}$ or $\beta_p\zeta^{\mu}$ is
real. Suppose that $\alpha_p\zeta^{\mu}=\nu\in\R$. Then $\overline{\alpha_p}\zeta^{-\mu}=\nu$, and by \thref{deligne} we have $\nu^2=|\alpha_p|^2= p^{k-1}$. Hence $\nu=\pm p^{(k-1)/2}$. It follows that
$$
a(p)=\alpha_p+\beta_p=\pm p^{(k-1)/2}(\zeta^{-\mu}+\chi(p)\zeta^{\mu}).
$$
We get the same result if we start with the condition that $\beta_p\zeta^\mu$ is real.

Suppose, for the sake of contradiction there are infinitely many primes $p$ for which there are integers $\mu_p\pmod j$
satisfying

\begin{equation}
a(p)=\pm p^{(k-1)/2}\upsilon_{\mu_p},\label{eq:7}
\end{equation}
where $\upsilon_{\mu_p}=\zeta^{-\mu_p}+\chi(p)\zeta^{\mu_p}$. We should note that 
\begin{equation}
\upsilon_{\mu_p}\neq 0,\label{eq.7}
\end{equation}
for all $\mu_p\in\{0,\cdots,j-1\}$, which is guaranteed by the assumptions $2\nmid j$, and $r_{\chi}$ is odd. Consider now 
$$
K_f:=\Q(\{a(p)\}_p),
$$
the subfield of $\C$ generated by all $a(p)$ ($p$ runs on primes). It is a well known fact that $K_f$ is a number field. It follows that $K_f(\zeta)$ is also a finite extension of $\Q$. Altogether from \eqref{eq:7} and \eqref{eq.7} we see

\begin{equation}
\sqrt{p}\in K_f(\zeta).\label{eq,7}
\end{equation}
By our hypothesis we infer that there exists an infinite sequence of primes $p_1<p_2<p_3\dots$  satisfying \eqref{eq,7}. Consequently
$$
\Q(\sqrt{p_1},\sqrt{p_2},\sqrt{p_3},\ldots)\subset K_f(\zeta).
$$

However, it is classical that the degree of the extension 
$$
\Q(\sqrt{p_1},\sqrt{p_2},\sqrt{p_3},\dots)/\Q
$$
is infinite, which give our contradiction. We have thus proved that for almost all primes $p$ the right-hand side of \eqref{eq:5} has no real poles.  

It remains to exclude the case (b) when \eqref{eq:2} converges everywhere. From \thref{thm1}, we see that for primes  $p$  not satisfying (a), the series \eqref{eq:2} converges everywhere, and particularly, it is an entire function in $s$. By (1) of \thref{lem} we see that $f$ is an eigenfunction of $T_j(p)$. Let $\lambda_j(p)$ be the corresponding eigenvalue, hence from (2) of \thref{lem} we get 
$$
\sum_{n\ge 0}a(p^{jn})X^{jn}=\frac{1}{1-\lambda_j(p)X^j+p^{j(k-1)}\chi^j(p)X^{2j}}\cdot
$$
The denominator on the right-hand side is a polynomial in $X^j$ of degree $2$, hence
it is non-constant and so has zeros. Setting $X=p^{-s}$, we obtain a contradiction.

\subsection{Proof of \texorpdfstring{\thref{thm3}}{Theorem 3}}

Let $\mathrm{Sh}_t(f)$ be the modular form associated to $f$ under the Shimura correspondence. According to \cite{shimura73,niwa75}, we have $\mathrm{Sh}_t(f)\in S_{2k}(N/2,\chi^2)$ and the $n$-th Fourier coefficient of $\mathrm{Sh}_t(f)$ is given by 

\begin{equation}
A_t(n)=\sum_{d|n}\chi_{t,N}(d)d^{k-1}a\left(\frac{n^2}{d^2}t\right),\label{eq:8}
\end{equation}
where $\chi_{t,N}$ denotes the character $\chi_{t,N}(d):=\chi(d)\left(\frac{(-1)^{k}N^{2}t}{d}\right)$. Furthermore, \eqref{eq:8} is equivalent to 

\begin{equation}
\sum_{n\ge1}\frac{a(tn^2)}{n^s}=\dfrac{1}{L(s-k+1,\chi_{t,N})}L(s,\mathrm{Sh}_t(f)),\label{eq:9}
\end{equation}
where $L(s,\chi_{t,N})$ is the Dirichlet $L$-function associated to $\chi_{t,N}$, and $L(s,\mathrm{Sh}_t(f))$ is the Hecke $L$-function associated to the cusp form $\mathrm{Sh}_t(f)$.
Notice that $\mathrm{Sh}_t(f)\neq 0$, which is guaranteed by the assumption $a(tn_0^2)\neq 0$.

For the sake of contradiction we assume that $a(tn^2)$ lies in $\W(\theta_1,\theta_2)$ for all but finitely many $n\ge 1$. Then by \thref{thm1}, we infer that the series in the left-hand side of \eqref{eq:9} either has a singularity at the real point of its line of convergence or converge everywhere. Further, by \thref{rmk1} we obtain 

\begin{equation}
\sigma_{\mathrm{c}}=\sigma_{\mathrm{ab}}=\sigma_{\mathrm{hol}},\label{eq:10}
\end{equation}
where $\sigma_{\mathrm{ab}}$, $\sigma_{\mathrm{c}}$, $\sigma_{\mathrm{hol}}$, denotes respectively the abscissa of absolute convergence, the abscissa of convergence and the abscissa of holomorphy of the series $\sum a(tn^2)n^{-s}$.  

Since $L(1,\chi_{t,N})\neq 0$, the function $L(s-k+1,\chi_{t,N})^{-1}$ is holomorphic in some neighborhood of $s=k$. Since $L(s,\mathrm{Sh}_t(f))$ is entire, we deduce that the series in the left hand side of \eqref{eq:9} is holomorphic in a region contained in the half-plane  $\re(s)<k$, hence by \eqref{eq:10} we obtain $\sigma_{\mathrm{ab}}< k$.

On the other hand, the series $\sum_{n\ge1}|a(tn^2)|$ diverge, since otherwise the series $\sum_{n\ge1}|a(tn^2)|n^{-s}$ 
converge for $\re(s)>0$, consequently the Dirichlet series associated to $L(s,\mathrm{Sh}_t(f))$ converge absolutely for $\re(s)>k$, which contradict the fact that its abscissa of absolute convergence is $k+1/2$ (see \cite[Lemma]{kohnen10}).

It follows by a classical fact about Dirichlet series that the abscissa of absolute convergence of the left-hand side of \eqref{eq:9} is given by 
$$
\sigma_{\mathrm{ab}}=\inf\left\{\sigma\in\R\; : \; \sum_{n\le N}|a(tn^2)|=O_{\sigma}( N^{\sigma})\right\}.
$$
Therefore, there exists $\varepsilon>0$ for which
$$
\sum_{n\le N}|a(tn^2)|=O_{\varepsilon}( N^{k-\varepsilon}).
$$
Now arguing as around the end of the proof of \cite[Theorem]{kohnen10} we get a contradiction with the fact that the Dirichlet series associated to $L(s,\mathrm{Sh}_t(f))$ has $k+1/2$ as the abscissa of convergence.

\subsection{Proof of \texorpdfstring{\thref{thm4}}{Theorem 4}}

By way of contradiction, suppose there are infinitely many primes $p\nmid N$ such that  $a(tp^{2\nu})$ lies in the wedge $\W(\theta_1,\theta_2)$ for all but finitely many $\nu\in\N$. So, by \thref{thm1} the series 

$$
\sum_{\nu\ge0}a(tp^{2\nu})p^{-\nu s},
$$
either (a) converge for all $s\in\C$ or (b) has a singularity at the real point of its line of convergence. 

Let $\mathrm{Sh}_t(f)$ the Shimura lift of $f$ with respect to $t$. Let $\lambda_p$ denote the $p$-th Hecke eigenvalue of $f$. Since 
$$
T(p)\mathrm{Sh}_t(f)=\mathrm{Sh}_t(T(p^2)f),
$$
it follows that the $p$-th Hecke eigenvalue of $\mathrm{Sh}_t(f)$ is $\lambda_p$, where $T(p^2)$ is the Hecke operator on $S_{k+1/2}(N,\chi)$ and $T(p)$ is the Hecke operator
on $S_{2k}(N/2,\chi^2)$. By \cite[Corolary 1.8]{shimura73} we have 

\begin{equation}
\sum_{\nu\ge0}a(tp^{2\nu})p^{-\nu s}=a(t)\frac{1-\chi_{t,N}(p)p^{k-1-s}}{1-\lambda_p p^{-s}+\chi^2(p)p^{2k-1-2s}},\label{eq:12}
\end{equation}
where $\chi_{t,N}:=\chi(.)\left(\frac{(-1)^{k}N^{2}t}{.}\right)$. The denominator of the right hand side of \eqref{eq:12} factorizes as follows 

\begin{equation*}
1-\lambda_p p^{-s}+\chi^2(p)p^{2k-1-2s}=(1-\alpha_p p^{-s})(1-\beta_p p^{-s}),
\end{equation*}
where $\alpha_p+\beta_p=\lambda_p,$ and $\alpha_p\beta_p=\chi^2(p)p^{2k-1}$. By \thref{deligne} we have

\begin{equation}
|\alpha_p|=p^{k-1/2},\quad|\beta_p|=p^{k-1/2}\cdot\label{eq:13}
\end{equation}
It is clear that the alternative (a) cannot occur, since the right-hand side of \eqref{eq:12} has a pole for $p^{s}=\alpha_p$ or $p^{s}=\beta_p.$
Thus the alternative (b) must hold, therefore $\alpha_p$ or $\beta_p$ must be real. Suppose that $\alpha_p\in\R$. By \eqref{eq:13} we have

\begin{equation}
\lambda_p=\alpha_p+\beta_p=\pm p^{k-1/2}(1+\chi^2(p)).\label{eq:14}
\end{equation}
Since $r_{\chi^2}$ is odd we have $1+\chi^2(p)\neq 0$. Hence $\sqrt{p}$ is contained in the number field $K_f$. Now we can derive a contradiction by arguing as around the end of the proof of \thref{thm2}. Therefore, the assumption that there are infinitely many primes for which the sequence $a(tp^{2\nu})$ lies in the wedge $\W(\theta_1,\theta_2)$, for all but finitely many $\nu\in\N$, must be false.

\section*{Acknowledgements}
The author would like to thank  the referee for his careful reading as well as Winfried Kohnen and Thomas A Hulse for useful discussions concerning their work.
\bibliographystyle{plain}
\bibliography{mybibfile}

\begin{thebibliography}{10}

\bibitem{bruin}
Jan~Hendrik Bruinier and Winfried Kohnen.
\newblock Sign changes of coefficients of half integral weight modular forms.
\newblock In B.~Edixhoven, van~der G.~Gerard, and B.~Moonen, editors, {\em
  Modular Forms on Schiermonnikoog}, pages 57--65. Cambridge University Press,
  2008.

\bibitem{Deligne73}
Pierre {Deligne}.
\newblock {La conjecture de {W}eil. I.}
\newblock {\em {Publ. Math., Inst. Hautes \'Etud. Sci.}}, 43:273--307, 1973.

\bibitem{Fekete}
M~Fekete.
\newblock Sur les s\' eries de dirichlet.
\newblock {\em Comptes rendus hebdomadaires des séances de l'Académie des
  sciences}, (150):1033--1036, 1910.

\bibitem{huls16}
Thomas~A. Hulse, Chan~Ieong Kuan, David Lowry-Duda, and Alexander Walker.
\newblock Sign {C}hanges of {C}oefficients and {S}ums of {C}oefficients of
  ${L}$--functions.
\newblock {\em J. Number Theory}, 177:112--135, 2017.

\bibitem{knopp03}
Marvin Knopp, Winfried Kohnen, and Wladimir Pribitkin.
\newblock On the signs of {F}ourier coefficients of cusp forms.
\newblock {\em Ramanujan J.}, 7(1):269--277, 2003.

\bibitem{kohnen10}
Winfried Kohnen.
\newblock A short note on {F}ourier coefficients of half-integral weight
  modular forms.
\newblock {\em Int. J. of Number Theory}, 6(06):1255--1259, 2010.

\bibitem{kohnen14}
Winfried Kohnen and Yves Martin.
\newblock Sign changes of {F}ourier coefficients of cusp forms supported on
  prime power indices.
\newblock {\em Int. J. of Number Theory}, 10(08):1921--1927, 2014.

\bibitem{maurizi11}
Brian~N. Maurizi.
\newblock Extending {L}andau's {T}heorem on {D}irichlet series with
  non-negative coefficients.
\newblock {\em Missouri J. Math. Sci.}, 23(2):105--122, 2011.

\bibitem{murty83}
M.~Ram Murty.
\newblock Oscillations of {F}ourier coefficients of modular forms.
\newblock {\em Math. Ann}, pages 431--446, 1983.

\bibitem{niwa75}
Shinji Niwa.
\newblock Modular forms of half integral weight and the integral of certain
  theta-functions.
\newblock {\em Nagoya Math. J.}, 56:147--161, 1975.

\bibitem{shimura73}
Goro Shimura.
\newblock On modular forms of half-integral weight.
\newblock {\em Annals of Mathematics}, 97(3):440--481, 1973.

\end{thebibliography}
\Addresses
\end{document}